\pdfoutput=1 

\documentclass[12pt,a4paper]{amsart}

\usepackage[T1]{fontenc}

\usepackage{amsthm}
\usepackage[top=35mm, bottom=35mm, left=30mm, right=30mm]{geometry}
\usepackage{microtype}
\usepackage[colorlinks=true,citecolor=blue]{hyperref}
\usepackage[looser]{newtxtext}
\usepackage{newtxmath}
\usepackage{enumerate}
\usepackage{verbatim}
\usepackage[all]{xy}
\usepackage{xcolor}

\newtheorem{thm}{Theorem}[section]
\newtheorem{cor}[thm]{Corollary}
\newtheorem{lem}[thm]{Lemma}
\newtheorem{prop}[thm]{Proposition}

\newtheorem{ques}[thm]{Question}

\theoremstyle{definition}

\newtheorem{rem}[thm]{Remark}
\numberwithin{equation}{section}

\newcommand{\bbn}{\mathbb{N}}
\newcommand{\eps}{\varepsilon}
\newcommand{\orb}{\mathrm{Orb}}
\newcommand{\syn}{\mathrm{syn}}
\DeclareMathOperator{\diam}{diam}
\DeclareMathOperator{\fs}{FS}
\DeclareMathOperator{\trans}{Trans}
\DeclareMathOperator{\eq}{Eq}
\newcommand{\ip}{\mathrm{IP}}

\title{On $n$-tuplewise IP-sensitivity and thick sensitivity}
\author{Jian Li and Yini Yang}
\address{Department of Mathematics, Shantou University,
	Shantou, Guangdong, 515063, P.R. China}
\email{lijian09@mail.ustc.edu.cn}
\email{ynyangchs@foxmail.com}

\date{\today}
\subjclass[2020]{37B05, 37B25}
\keywords{Sensitivity, equicontinuity, IP-set, thick set, minimal system, distal factor}

\begin{document}

\begin{abstract}
Let $(X,T)$ be a topological dynamical system and $n\geq 2$. We say that $(X,T)$ is $n$-tuplewise IP-sensitive (resp. $n$-tuplewise thickly sensitive) if there exists a constant $\delta>0$ with the property that for each non-empty open subset $U$ of $X$,
there exist $x_1,x_2,\dotsc,x_n\in U$  such that
\[
\Bigl\{k\in\mathbb{N}\colon \min_{1\le i<j\le n}d(T^k x_i,T^k x_j)>\delta\Bigr\}
\]
is an IP-set (resp. a thick set).

We obtain several sufficient and necessary conditions of a dynamical system to be $n$-tuplewise IP-sensitive or $n$-tuplewise thickly sensitive and show that any non-trivial weakly mixing system is $n$-tuplewise IP-sensitive for all $n\geq 2$, while it is $n$-tuplewise thickly sensitive if and only if it has at least $n$ minimal points.
We characterize two kinds of sensitivity
by considering some kind of factor maps.
We introduce the opposite side of pairwise IP-sensitivity and pairwise thick sensitivity, named (almost) pairwise IP$^*$-equicontinuity and (almost) pairwise syndetic equicontinuity, and obtain dichotomies results for them. In particular, we show that a minimal system is distal if and only if it is pairwise IP$^*$-equicontinuous. 
We show that every minimal system admits a maximal almost pairwise IP$^*$-equicontinuous factor and admits a maximal pairwise syndetic equicontinuous factor, and characterize them by the factor maps to their maximal distal factors.
\end{abstract}

\maketitle
	
\section{Introduction}

Throughout this paper, by a \textit{(topological) dynamical system} we mean a pair $(X, T)$, where $X$ is a compact metric space with a metric $d$ and $T\colon X \to X$ is a continuous surjective map.

A dynamical system $(X,T)$ is called \emph{equicontinuous} if for any $\eps>0$ there exists $\delta>0$ such that for any $x,y\in X$ with $d(x,y)<\delta$ one has $d(T^k x,T^k y)<\eps$ for all $k\in\bbn$.
Roughly speaking, in an equicontinuous system if two points are close enough then they will always be close under iterations. So equicontinuous systems have simple dynamical behaviors.
The opposite side of equicontinuity is sensitive dependence on initial conditions (or briefly sensitivity).
To be precise, a dynamical system $(X,T)$ is called \textit{sensitive} if there is a constant $\delta>0$ with the property that for any opene (open and non-empty) subset $U$ of $X$ there exist $x,y\in U$ and $k\in \bbn$ such that $d(T^k x,T^k y)>\delta$.
The celebrated  dichotomy theorem proved by Auslander and Yorke \cite{AY80} is as follows: every minimal system is either equicontinuous or sensitive.

The requirement to be equicontinuous is too strong, while to be sensitive is relatively weak.
It is natural to consider weak forms of equicontinuity and strong forms of sensitivity.
In \cite{F51} Fomin introduced a weak form of equicontinuity, called stability in the mean in the sense of Lyapunov, which means if two points are close enough then they will always be close under iterations except a set with small upper density.
In \cite{C63} Clay introduced two variations of equicontinuity related to syndetic sets and showed that both of them share similar properties of equicontinuity in some sense.
In \cite{A97} and \cite{F81} the authors studied many dynamical properties via collections of subsets of positive integers with special combinatorial properties. Recently the study of sensitivity along subsets of positive integers has attracted lots of attention, see e.g. \cite{HKZ18, LTY15, LY20, M07, TZ11, YY18}.
We refer the reader to the surveys \cite{LY16} and \cite{LYY21} for related results.

In \cite{X05} Xiong introduced a multi-variable form of sensitivity, named $n$-sensitivity for $n\geq 2$, that is, a dynamical system $(X,T)$ is called \emph{$n$-sensitive} if there exists a constant $\delta>0$ such that for any opene set $U\subset X$, there exist $x_1, x_2, \dotsc, x_n\in U$ and $k\in \bbn$ with
\[\min_{1\le i<j\le n}d(T^k x_i,T^k x_j)>\delta.\]

In this paper, we extend some results on strongly IP-sensitive and strongly thickly sensitive in \cite{YY18} to $n$-tuplewise IP-sensitive and $n$-tuplewise thickly sensitive.
More specifically,
we say that a system $(X,T)$ is $n$-tuplewise IP-sensitive (resp.\ $n$-tuplewise thickly sensitive)
if there exists a constant $\delta>0$ with the property that for each non-empty open subset $U$ of $X$
there exist $x_1,x_2,\dotsc,x_n\in U$ such that
\[
\Bigl\{k\in\bbn\colon \min_{1\le i<j\le n}d(T^k x_i,T^k x_j)>\delta\Bigr\}
\]
is an IP-set (resp.\ a thick set).
We obtain some sufficient and necessary conditions of a dynamical system to be $n$-tuplewise IP-sensitive or $n$-tuplewise thickly sensitive.
We characterize two kinds of sensitivity
by considering some kind of factor maps.  (see Lemma~\ref{lem:IP-sen-semi-open} and~\ref{lem:semi-open-thick}, Theorem~\ref{thm:proximal-not-1-1-IP}, Proposition~\ref{prop:pi-proximal-n-thick} and Corollary~\ref{cor:proximal-factor-thick-sen}).
We also obtain some new results on the sensitivity of weakly mixing systems.
It is shown that any non-trivial weakly mixing system is $n$-tuplewise IP-sensitive for all $n\geq 2$ (see Theorem~\ref{thm:weakly-mixing-IP-sen}), while it is $n$-tuplewise thickly sensitive if and only if it has at least $n$ minimal points (see Theorem~\ref{thm:w-mixing-thick-AP}).
Using the new results we find some examples to show that there exists a system which is $n$-tuplewise thickly sensitive but not
$(n+1)$-tuplewise thickly sensitive, and there exists a system which is  $n$-tuplewise IP-sensitive for all $n\geq 2$ but not pairwise thickly sensitive (see Remark~\ref{rem:thick-exam} and Remark~\ref{rem:thick-ip-exam}).

Motivated by the celebrated dichotomy theorem on equicontinuity and sensitivity,
we also investigate the opposite side of pairwise IP-sensitivity and pairwise thick sensitivity, named (almost) pairwise IP$^*$-equicontinuity and (almost) pairwise syndetic equicontinuity, and obtain dichotomies results for them (see Theorems~\ref{thm:minimal-dichotomy-IP} and~\ref{thm:trans-dichotomy-thick}, Corollary~\ref{cor:minimal-dichotomy-thick}).
It should be noticed that the collections of IP-sets and IP$^*$-sets are not translation invariant. Unlike the standard technique, we need the structure theorem
(\cite[Theorem C]{YY18}) of pairwise IP-sensitivity to get the dichotomy result (see Theorem~\ref{thm:minimal-dichotomy-IP}).
We show that every minimal system admits a maximal almost pairwise IP$^*$-equicontinuous factor which can be regarded as the maximal almost one-to-one extension of its maximal distal factor (see Theorem~\ref{thm:max-ip-*}),
and admits a maximal pairwise syndetic equicontinuous factor which
can be regarded as the maximal proximal extension of its maximal distal factor (see Theorem~\ref{thm:max-syndetic}),
while how to characterize the invariant closed equivalence relations generated by the maximal factors is still open.

It used to be an open question in 1950s that whether every minimal distal system is equicontinuous. This question was answered negatively in \cite{AGH63} and \cite{F61} in early of 1960s.
As an application, we show that one can characterize the distal property via a weaker form of equicontinuity.
To be precise, we show that a minimal system is distal if and only if it is pairwise IP$^*$-equicontinuous (see Corollary~\ref{cor:min-distal-ip-star-eq}).

The paper is organized as follows.
In Section $2$, we recall some definitions and
related results which will be used later.
In Sections $3$ and $4$, we study $n$-tuplewise IP-sensitivity and $n$-tuplewise thick sensitivity respectively.

\section{Preliminaries}
	
In this section we will recall some basic notions and results which are needed in the following sections. We refer the reader to \cite{A97}, \cite{A88} and \cite{F81} for more details on topological dynamics.

Let $\bbn$ denote the collection of positive integers.
A subset $F$ of $\bbn$ is called \emph{thick} if for every $L\in\bbn$ there exists $m_L\in\bbn$ such that $\{m_L,m_L+1,\dotsc,m_L+L\}\subset F$,
and \emph{syndetic} if there is an $M\in\mathbb{N}$ such that $F\cap\{m,m+1,\ldots,m+M\}\neq \emptyset$ for all $m\in \bbn$.
It is clear that a subset $F$ of $\bbn$ is thick if and only if  it has a non-empty intersection with any syndetic set.
A subset $F$ of $\bbn$ is said to be an \emph{IP-set} if there exists an sequence $\{p_i\}_{i=1}^\infty$ such that $\fs(\{p_i\}_{i=1}^\infty)\subset F$,
where
\[
\fs(\{p_i\}_{i=1}^\infty)= \Bigl\{\sum_{i\in \alpha} p_i \colon \alpha
\text{ is a  non-empty finite  subset of }  \bbn \Bigr\}.
\]
The collection of IP-sets has the Ramsey property, that is, if $F$ is an IP-set and $F=F_1\cup F_2$ then either $F_1$ or $F_2$ is an IP-set.
This result is due to Hindman, see e.g. \cite[Proposition~8.13]{F81}.
A subset $F$ of $\bbn$ is  said to be an \emph{IP$^*$-set}
if  it has a non-empty intersection with any IP-set.
The collection of IP$^*$-sets has the filter property, that is,
if $F_1$ and $F_2$ are IP$^*$-sets then so is $F_1\cap F_2$.

Let $(X,T)$ be a dynamical system and $n\in\bbn$.
The $n$-fold product system of $(X,T)$ is denoted by $(X^n,T^{(n)})$, where $T^{(n)}(x_1,\dotsc,x_n)=(T x_1,\dotsc,Tx_n)$ for any $(x_1,\dotsc,x_n)\in X^n$.
For a point $x\in X$,
the \emph{orbit} of $x$ is the set $\{T^k x\colon k\geq 0\}$, denoted by
$\overline{\orb(x,T)}$, and the \emph{$\omega$-limit set of $x$}, denoted by
$\omega(x,T)$, is the collection of limit points of the sequence $\{T^kx\}_{k=1}^\infty$.

If a non-empty closed subset $Y\subset X$ is $T$-invariant, i.e. $TY\subset Y$, then the restriction $(Y, T|_Y)$ is also a  dynamical system, which is referred to be a \textit{subsystem} of $(X,T)$. We will write $(Y,T)$  instead of $(Y, T|_Y)$ for simplicity.
It is clear that for any $x\in X$, $\overline{\orb(x,T)}$ and $\omega(x,T)$ are closed and $T$-invariant.
If a dynamical system $(X,T)$ has no proper subsystems, then we say that it is \emph{minimal}.

A point $x\in X$ is called \emph{recurrent} if $x\in\omega(x,T)$, \emph{transitive} if $\omega(x,T)=X$
and \emph{minimal} if $\bigl(\overline{\orb(x,T)},T\bigr)$ is a minimal
subsystem of $(X,T)$.
A dynamical system $(X,T)$ is called \emph{transitive} if it has a transitive point. Denoted by $\trans(X,T)$ the collection of all transitive points in $X$.
It is not hard to see that a dynamical system $(X,T)$ is minimal if and only if $\trans(X,T)=X$, and in a transitive system $(X,T)$, $\trans(X,T)$ is a dense $G_\delta$ subset of $X$.
A dynamical system $(X,T)$ is called \emph{weakly mixing}
if the product system $(X^2,T^{(2)})$ is transitive.
We say that a dynamical system $(X,T)$ is non-trivial if $X$ is not a singleton.
It is clear that if a weakly mixing system $(X,T)$ is non-trivial then $X$ is perfect.

Let $x\in X$ and $U,V\subset X$. Define
\[
N(x,U)=\{k\in\bbn \colon T^k x\in U\} \ \text{ and }\  N(U,V)=\{k\in\bbn\colon U\cap T^{-k}V\neq\emptyset\}.
\]

A point $x\in X$ is recurrent (resp. minimal) if and only if for every neighborhood $U$ of $x$, $N(x,U)$ is an IP-set (resp. $N(x,U)$ is syndetic).
A dynamical system $(X,T)$ is weakly mixing if and only if for every two opene subsets $U$ and $V$ of $X$, $N(U,V)$ is thick.
We say that a dynamical system is \emph{strongly mixing} if
for every two opene subsets $U$ and $V$ of $X$, $N(U,V)$ is cofinite, that is, there exists $N\in\bbn$ such that $\{N,N+1,N+2,\dotsc\}\subset N(U,V)$.

\begin{lem}[{\cite[Proposition II.3]{F67}}]\label{lem:weakly-mixing-n}
If $(X,T)$ is a weakly mixing system, then for any $n\geq 2$, the $n$-fold product system $(X^{n}, T^{(n)})$ is transitive.
\end{lem}	

A pair $(x,y)\in X^2$ is called \emph{proximal}
if
\[
\liminf_{k\to\infty}d(T^k x,T^k y)=0.
\]
Let $P(X,T)$ denote the collection of all proximal pairs in $(X,T)$.
Then $P(X,T)$ is a reflexive symmetric $T$-invariant relation on $X$, but is neither transitive nor closed in general.

\begin{lem}[{\cite[Lemma~2]{A60}}]
\label{lem:orbit-minimal-proximal}
Let $(X,T)$ be a dynamical system and $x\in X$. Then for any minimal subset $M$ of $\overline{\orb(x,T)}$, there exists a point $y\in M$ such that $(x,y)$ is proximal.
\end{lem}

Let $(X,T)$ be a dynamical system and $x,y\in X$.
We say that $x$ is \emph{strongly proximal} to $y$ if $(y,y)\in \omega((x,y),T\times T)$, where $\omega((x,y),T\times T)$
is the $\omega$-limit set of $(x,y)$.
Note that if $(x,y)$ is proximal and $y$ is a minimal point then $x$ is strongly proximal to $y$.
	
\begin{lem}[\mbox{\cite[Lemma 4.8]{L12}}] \label{lem:strongly-proximal}
Let $(X,T)$ be a dynamical system and $x,y\in X$.
Then $x$ is strongly proximal to $y$ if and only if
for every neighborhood $U$ of $y$, $N(x,U)\cap N(y,U)$ is an IP-set.
\end{lem}
	
\begin{prop}[\mbox{\cite[Proposition 5.9]{L12}}] \label{prop:IP-set-stronlgy-proximal}
Let $(X,T)$ be a dynamical system, $x\in X$ and $Y\subset X$ be a closed subset of $X$.
If $N(x,Y)$ is an IP-set, then there exists a point $y\in Y$ such that $x$ is strongly proximal to $y$.
\end{prop}

Following \cite{A97}, we say that a dynamical system $(X,T)$ is \emph{IP$^*$-central}
if for any opene subset $U$ of $X$, $N(U,U)$ is an IP$^*$-set.
If a dynamical system $(X,T)$ admits an invariant measure with full support, then by the well-known Pioncar\'e recurrent theorem it is easy to see that $(X,T)$ is IP$^*$-central, in particular a minimal system is IP$^*$-central.

\begin{lem}\label{lem:supp-IP=subset}
	If a dynamical system $(X,T)$ is IP$^*$-central,
	then for any IP-set $F$ and opene subset $U$ of $X$, there exist
	an IP-subset $F^{\prime}$ of $F$ and a point $z\in U$ such that $F^{\prime}\subset N(z,U)$.
\end{lem}

\begin{proof}
	Fix an IP-set $F$ and an opene subset $U$ of $X$.
	Pick a sequence $\{p_i\}_{i=1}^{\infty}$ in $\bbn$ with $p_{i+1}>\sum_{j=1}^{i}p_j$ such that $FS\{p_i\}_{i=1}^{\infty} \subset F$.
	Take an opene subset $V_0$ of $X$ such that $\overline{V_0}\subset U$.
	Since $(X,T)$ is IP$^*$-central, there is a finite subset $\alpha_1$ of $\bbn$ such that $V_0\cap T^{-p_{\alpha_1}}V_0\neq\emptyset$ where $p_{\alpha_1}=\sum_{j\in\alpha_1}p_j$.
	Take an opene $V_1$ of $X$ such that $\overline{V_1}\subset V_0\cap T^{-p_{\alpha_1}}V_0$,
	there is a finite subset $\alpha_2$ of $\bbn$ with $\min{\alpha_2}>\max{\alpha_1}$ and
	$V_1\cap T^{-p_{\alpha_2}}V_1\neq \emptyset$.
	By induction we get a sequence $\{\alpha_i \}$ of finite subsets of $\bbn$ and a sequence of opene sets $\{V_i\}$ which satisfy  that for any $i\in\bbn$,
	$\min{\alpha_{i+1}}>\max{\alpha_i}$ and
	$\overline{V_{i+1}}\subset V_i\cap T^{-p_{\alpha_{i+1}}}V_i$.
	By the compactness of $X$, take a point $z\in\bigcap_{i=1}^{\infty}\overline{V_i}$ and let $q_i=p_{\alpha_i}$ for $i\in\bbn$. It is easy to verify that $FS\{q_i\}_{i=1}^{\infty}\subset N(z,U)$.
\end{proof}

A pair $(x,y)\in X^2$ is called \emph{distal} if it is not proximal.
A point $x\in X$ is called \emph{distal} if for any $y\in \overline{\orb(x,T)}\setminus\{x\}$, $(x,y)$ is distal.
By Lemma~\ref{lem:orbit-minimal-proximal}, every distal point is minimal. Furthermore, by \cite[Theorem 9.11]{F81} a point $x$ is distal if and only if for every neighborhood $U$ of $x$, $N(x,U)$ is an IP$^*$-set.
A dynamical system $(X,T)$ is called \emph{distal} if any pair
of distinct points in $(X,T)$ is distal.
It is clear that a dynamical system is distal if and only if every point is distal.
A minimal system $(X,T)$ is called \emph{point-distal}
if there exists some distal point in $X$.

Let $(X,T)$ and $(Y,S)$ be two dynamical systems.
If there is a continuous surjection $\pi\colon X \to Y$ such that $\pi\circ T = S\circ \pi$, then we say that $\pi$ is a \emph{factor map}, the system $(Y,S)$ is a \emph{factor} of $(X,T)$ or $(X, T)$ is an \emph{extension} of $(Y,S)$.
For a factor map $\pi\colon (X,T)\to (Y,S)$, let
\[
R_\pi=\{(x,y)\in X^2 \colon \pi(x)=\pi(y)\}.
\]
Then $R_\pi$ is a $T\times T$-invariant closed equivalence relation on $X$ and $Y=X/R_\pi$.
In fact, there exists a one-to-one correspondence between the collection of factors of $(X,T)$ and the collection of $T\times T$-invariant closed  equivalence relations on $X$.

Every topological dynamical system $(X,T)$ admits a maximal distal factor $(X_d,T_d)$, that is, every distal factor of $(X,T)$ is also a factor of $(X_d,T_d)$.
The maximal distal factor $(X_d,T_d)$ corresponds the smallest $T\times T$-invariant closed equivalence relation on $X$ containing the proximal relation $P(X,T)$ (see \cite[Theorem 2]{EG60}).

\begin{lem}[{\cite[Corollary 1]{A60} or \cite[Theorem 3]{C63b}}] \label{lem:proximal-closed}
Let $(X,T)$ be a dynamical system.
If $P(X,T)$ is closed, then it is an equivalence relation on $X$.
\end{lem}
	
A factor map $\pi\colon (X, T) \to (Y, S)$ is called \emph{proximal} if $R_{\pi}\subset P(X,T)$ and \emph{almost one-to-one} if $\{x \in X \colon \pi^{-1}(\pi(x)) \textrm{ is a singleton} \}$ is residual in $X$.
For a factor map $\pi\colon (X,T)\to (Y,S)$ between minimal systems, $\pi$ is almost one-to-one if and only if there exists a point $y\in Y$ such that $\pi^{-1}(y)$ is a singleton (see e.g. \cite[Corollary 3.6]{LY20}).
Let $X$ and $Y$ be compact metric spaces.
A map $\pi\colon X\to Y$ is called \emph{semi-open} if for every opene subset $U$ of $X$, $\pi(U)$ has a non-empty interior.
It is easy to see that a map $\pi\colon X\to Y$ is semi-open if and only if for every dense subset $A$ of $Y$, $\pi^{-1}(A)$ is dense in $X$.
Note that any factor map $\pi\colon (X,T)\to (Y,S)$ between minimal systems is semi-open (see e.g. \cite[Theorem~1.15]{A88}).

\section{\texorpdfstring{$N$}{N}-tuplewise IP-sensitivity}
	
Let $(X,T)$ be a dynamical system and $n\geq 2$.
We say that $(X,T)$ is \emph{$n$-tuplewise IP-sensitive}
if there exists a constant $\delta>0$ with the property that for any opene subset $U$ of $X$ there exist $x_1,\dotsc,x_n\in U$ such that
\[
	\Bigl \{ k\in\bbn\colon \min_{1\le i<j\le n}d(T^k x_i,T^k x_j)>\delta\Bigr \}
\]
is an IP-set. The constant $\delta$ is called an \emph{$n$-tuplewise IP-sensitive constant} for $(X,T)$.
For simplicity, $2$-tuplewise IP-sensitivity will be called \emph{pairwise IP-sensitivity}.

We first study some basic properties of \texorpdfstring{$n$}{n}-tuplewise IP-sensitivity.

\subsection{Properties of \texorpdfstring{$n$}{n}-tuplewise IP-sensitivity}

We show a sufficient and necessary condition related to strong proximality for $n$-tuplewise IP-sensitivity.
	
\begin{prop}\label{prop:strongly-proximal-IP-n}
    Let $(X,T)$ be a dynamical system and $n\geq 2$.
    Then $(X,T)$ is $n$-tuplewise IP-sensitive if and only if
    there exists a constant $\delta>0$ with the property that for
    any opene subset $U$ of $X$,
    there exist $x_1,\dotsc,x_n\in U$ and $y_1,\dotsc,y_n\in X$ such that $\min_{1\le i<j\le n}d(y_i,y_j)\geq\delta$ and $(x_1,\dotsc, x_{n})$ is strongly proximal
	to $(y_1,\dotsc, y_{n})$ in the $n$-fold product system $(X^n,T^{(n)})$.
\end{prop}
\begin{proof}
	($\Leftarrow$)
	Let $\delta$ be the constant with the desired property.
	Fix an opene subset $U$ of $X$.
	Then $x_1,\dotsc,x_n\in U$ and $y_1,\dotsc,y_n\in X$ such that $\min_{1\le i<j\le n}d(y_i,y_j)\geq\delta$ and $(x_1,\dotsc, x_{n})$ is strongly proximal
	to $(y_1,\dotsc, y_{n})$.
	Pick a neighborhood $V$ of $(y_1,\dotsc,y_n)$ in $X^n$
	such that for any $(z_1,\dotsc,z_n)\in V$, one has
	$\min_{1\le i<j\le n}d(z_i,z_j)>\frac{\delta}{2}$.
	According to Lemma~\ref{lem:strongly-proximal},
	$N((x_1,\dotsc, x_{n}), V)$ is an IP-set.
	Then $\frac{\delta}{2}$ is an $n$-tuplewise IP-sensitive
	constant for $(X,T)$.
	
	$(\Rightarrow)$
	Let $\delta$ be an $n$-tuplewise IP-sensitive constant for $(X,T)$. Fix an opene set $U$ of $X$, there exist
	$x_1,\dotsc, x_{n}\in U$ such that
	\[
	F:=\Bigl\{k\in\bbn\colon \min_{1\le i<j \le n}d(T^k x_i,T^k x_j)>\delta\Bigr \}
	\]
	is an IP-set.
	According to Proposition~\ref{prop:IP-set-stronlgy-proximal},  there exists a point
	\[
	(y_1,\dotsc, y_{n})\in\overline{ \{T^{(n)}(x_1,\dotsc,x_n)\colon n\in F\}}
	\]
	such that $(x_1,\dotsc, x_{n})$ is strongly proximal to $(y_1,\dotsc, y_{n})$. By the construction of
	$F$, one has
	$\min_{1\le i<j\le n}d(y_i,y_j)\geq\delta$.
	Then the constant $\delta$ has the desired property.
\end{proof}

The following result reveals that a non-trivial weakly mixing system is $n$-tuplewise IP-sensitive for all $n\geq 2$ by using Proposition~\ref{prop:strongly-proximal-IP-n}.

\begin{thm}\label{thm:weakly-mixing-IP-sen}
    If $(X,T)$ is a non-trivial weakly mixing system,
    then $(X,T)$ is $n$-tuplewise IP-sensitive for all $n\geq 2$.
\end{thm}
\begin{proof}
As the weak mixing system $(X,T)$ is non-trivial, the compact metric space $X$ is perfect.
Fix $n\geq 2$ and pick a point $(z_1,\dotsc,z_n)\in X^n$ with pairwise distinct coordinates.
Let $\delta= \frac{1}{2}\min_{1\le i<j \le n}d(z_i,z_j)$. Then $\delta>0$.	
Pick a neighborhood $V$ of $(z_1,\dotsc,z_n)$ in $X^n$ such that for any $(z_1^{\prime},\dotsc,z_n^{\prime})\in V$, one has $\min_{1\le i<j\le n}d(z_i^{\prime},z_j^{\prime})>\delta$.
Since $(X,T)$ is weakly mixing, by Lemma \ref{lem:weakly-mixing-n} $(X^{n}, T^{(n)})$ is also weakly mixing.
For any opene set $U\subset X$, there exist $x_1,\dotsc,x_n\in U$ and $(y_1,\dotsc,y_n)\in V$ such that $(x_1,\dotsc,x_n, y_1,\dotsc,y_n)$ is a transitive point of $(X^{2n}, T^{(2n)})$.
Then $\min_{1\le i\not=j\le n}d(y_i,y_j)>\delta$ and
\[
(y_1,\dotsc,y_n, y_1,\dotsc,y_n)\in \omega\bigl((x_1,\dotsc,x_n, y_1,\dotsc,y_n),T^{(2n)}\bigr),
\]
which shows that $(x_1,\dotsc,x_n)$ is strongly proximal to $(y_1,\dotsc,y_n)$.
By Proposition~\ref{prop:strongly-proximal-IP-n}, $(X,T)$ is $n$-tuplewise IP-sensitive.
\end{proof}
			
For IP$^*$-central systems,
the sufficient and necessary condition for $n$-tuplewise IP-sensitivity can be  characterized as follows.
	
\begin{prop}\label{prop:strongly-proximal-IP}
Let $(X,T)$ be an IP$^*$-central system and $n\geq 2$.
Then $(X,T)$ is $n$-tuplewise IP-sensitive if and only if
there exists a constant $\delta>0$ with the property that for any opene subset $U$ of $X$,
there exist $x_1,\dotsc,x_{n-1}\in U$ and $y_1,\dotsc,y_{n-1}\in X$ with $\min_{1\le i<j\le n-1}d(y_i,y_j)>\delta$ and $\min_{1\le i\le n-1}d(x_i, y_i)>\delta$ such that $(x_1,\dotsc, x_{n-1})$ is strongly proximal to $(y_1,\dotsc, y_{n-1})$ in the $(n-1)$-fold product system $(X^{n-1},T^{(n-1)})$.
\end{prop}

\begin{proof}
($\Leftarrow$)
Let $\delta>0$ be a constant with the desired property.
Fix an opene subset $U$ of $X$. Without loss of generality we assume that $\diam(U)<\frac{\delta}{3}$.
There exists $x_1,\dotsc,x_{n-1}\in U$ and $y_1,\dotsc,y_{n-1}\in X$, with $\min_{1\le i<j\le n-1}d(y_i,y_j)>\delta$ and $\min_{1\le i\le n-1}d(x_i, y_i)>\delta$ such that $(x_1,\dotsc, x_{n-1})$ is strongly proximal to $(y_1,\dotsc, y_{n-1})$.
Let $V=\{(z_1,\dotsc, z_{n-1})\in X^{n-1}\colon d(y_i,z_i)<\frac{\delta}{3},i=1,\dotsc,n-1\}$.
By Lemma~\ref{lem:strongly-proximal},
$F:=N((x_1,\dotsc, x_{n-1}),V)$ is an IP-set.
By Lemma~\ref{lem:supp-IP=subset}, there exists a point $x_n\in U$ and an IP-subset $F'$ of $F$ such that $F'\subset N(x_n,U)$.
For every $k\in F'$, $d(T^k x_i,y_i)<\frac{\delta}{3}$ for $i=1,\dotsc,n-1$ and  $T^k x_n\in U$.
Note that $\max_{1\leq i\leq n-1}d(T^k x_n, x_i)<\frac{\delta}{3}$,
$\min_{1\le i<j\le n-1}d(y_i,y_j)>\delta$ and $\min_{1\le i\le n-1}d(x_i, y_i)>\delta$, one has $\min_{1\le i<j \le n}d(T^kx_i,T^kx_j)>\frac{\delta}{3}$.
This implies that
\[
F'\subset \Bigl \{k\in\bbn\colon \min_{1\le i<j \le n}d(T^kx_i,T^kx_j)>\frac{\delta}{3}\Bigr\}.
\]
Then $\frac{\delta}{3}$ is an  $n$-tuplewise IP-sensitive constant for $(X,T)$.
		
($\Rightarrow$)
Let $\delta$ be an $n$-tuplewise IP-sensitive constant.
Fix an opene subset $U$ of $X$.
Take $z\in U$ and let $V=U\cap B(z,\frac{\delta}{4})$.
By Proposition~\ref{prop:strongly-proximal-IP-n}, there exist $x_1,\dotsc,x_n\in V$ and $y_1,\dotsc,y_n\in X$ such that $\min_{1\le i<j\le n}d(y_i,y_j)\geq\delta$ and $(x_1,\dotsc, x_{n})$ is strongly proximal to $(y_1,\dotsc, y_{n})$.
Without loss of generality we assume that $y_1,\dotsc, y_{n-1}\not\in B(z,\frac{\delta}{2})$ since there exists at most one of them in $B(z,\frac{\delta}{2})$.
Thus we have $\min_{1\le i<j\le n-1}d(x_i, y_i)>\frac{\delta}{4}$ and it is easy to see that $\frac{\delta}{4}$ is the constant as required.	
\end{proof}
	
The case $n=2$ of Proposition~\ref{prop:strongly-proximal-IP} can be regarded as an improvement of \cite[Proposition 4.15]{YY18}, which states as follows.

\begin{cor}
If a dynamical system $(X,T)$ is IP$^*$-central, then
$(X,T)$ is pairwise IP-sensitive if and only if there is $\delta>0$  with the property that for any opene subset $U$ of $X$, there exist $x\in U$ and $y\in X$, with $d(x,y)>\delta$ such that $x$ is strongly proximal to $y$.
\end{cor}

\subsection{\texorpdfstring{$N$}{N}-tuplewise IP-sensitivity among factor maps}
We first show that $n$-tuplewise IP-sensitivity can be lifted by semi-open factor maps.

\begin{lem}\label{lem:IP-sen-semi-open}
	Let $\pi\colon (X,T)\to (Y,S)$ be a factor map and $n\geq 2$.
	If $(Y,S)$ is $n$-tuplewise IP-sensitive and $\pi$ is semi-open,
	then $(X,T)$ is also $n$-tuplewise IP-sensitive.
\end{lem}
\begin{proof}
	Choose an $n$-tuplewise IP-sensitive constant $\delta$ for $(Y,S)$.
	Since $\pi$ is uniformly continuous, there is $\eta>0$
	such that for any $u,v\in X$, $d(u,v)\le \eta$ implies that
	$d(\pi(u),\pi(v))\le \delta$.
	
	Fix an opene subset $U$ of $X$.
	As $\pi$ is semi-open, $int(\pi(U))$ is not empty.
	There exist $y_1,\dotsc,y_n\in int(\pi(U))$ such that
	\[
	F:=\Bigl\{k\in \bbn \colon \min_{1\le i<j\le n}d(S^k y_i, S^k y_j)>\delta\Bigr\}
	\]
	is an IP-set.
	Pick $x_i\in U$ such that $\pi(x_i)=y_i$ for $i=1,\dotsc,n$.
	By the choice of $\eta$, we have
	\[
	F\subset \Bigl\{k\in \bbn\colon \min_{1\le i<j\le n}d(T^k x_i, T^k x_j)>\eta\Bigr\}.
	\]
	This implies that $(X,T)$ is $n$-tuplewise IP-sensitive.
\end{proof}

We have the following sufficient condition for $n$-tuplewise IP-sensitivity on minimal systems, which generalizes the result of \cite[Proposition 4.7]{YY18}.

\begin{thm} \label{thm:proximal-not-1-1-IP}
Let $\pi:(X,T)\to (Y,S)$ be a factor map between two minimal systems.
If $\pi$ is proximal but not almost one-to-one then $(X,T)$ is $n$-tuplewise IP-sensitive for any $n\ge 2$.
\end{thm}
\begin{proof}
First we claim that $\pi^{-1}(y)$ is infinite for any $y\in Y$. Assume on the contrary there exists a point $y_0\in Y$ such that $\pi^{-1}(y_0)$ is finite.
Since $\pi$ is proximal, by \cite[Proposition 3.8, Corollary 3.9]{LY20},
there exists a point $y_1\in Y$ such that $\pi^{-1}(y_1)$ is a singleton, which contradicts to $\pi:X\to Y$ is not almost one-to-one.

Fix $n\geq 2$ and define
\begin{align*}
	\phi_n\colon Y &\to\mathbb{R}\\
	y&\mapsto \sup\Bigl\{\min_{1\le i<j\le n} d(x_i,x_j)\colon x_1,x_2,\dotsc,x_n\in \pi^{-1}(y)\Bigr\}.
\end{align*}
Then $\phi_n(y)>0$ for any $y\in Y$.
By \cite[Proposition 3.4]{LY20}, $\inf_{y\in Y} \phi_n(y)>0$.
Let $\delta=\frac{1}{4} \inf_{y\in Y} \phi_n(y)$.
For any opene set $U\subset X$, $int(\pi(U))\neq \emptyset$.
Since $\pi:(X,T)\to (Y,S)$ is proximal, the factor map $\pi:(X^{n},T^{(n)})\to (Y^{n},S^{(n)})$ is also proximal (see e.g. \cite[Lemma 4.3]{HKZ18}).
Choose  a point $y_0\in int(\pi(U))$. Since $\phi_n(y_0)> 3\delta>0$, we can find $u_1, \dotsc, u_n \in \pi^{-1}(y_0)$ such that $\min_{1\le i< j\le n} d(u_i,u_j)>3\delta$.
Let $W_i = B(u_i,\delta) \cap \pi^{-1}int(\pi (U))$, $i=1,\dotsc,n$.
Then $W_1,\dotsc, W_n$ are opene subsets of $X$ with
\[
\min_{1\le i<j\le n} d(W_i,W_j)>\delta.
\]
Since $(X,T)$ is minimal, $(X^{n},T^{(n)})$ has a dense set of minimal points.
Choose a minimal point $(y_1,\dotsc,y_n)\in W_1\times \dotsb \times W_n$ and points $x_i\in U$ with $\pi(x_i)=\pi(y_i)$ for $i=1,\dotsc,n$.
Since $(x_1,\dotsc,x_n)$ and $(y_1,\dotsc,y_n)$ are proximal with $(y_1,\dotsc,y_n)$ minimal point, $(x_1,\dotsc,x_n)$ is strongly proximal to $(y_1,\dotsc,y_n)$.
Therefore by Proposition~\ref {prop:strongly-proximal-IP-n}  $(X,T)$ is $n$-tuplewise IP-sensitive.
\end{proof}

It is shown in \cite[Theorem~C]{YY18} that
a minimal system $(X,T)$ is pairwise IP-sensitive if and only if
$\pi$ is not almost one-to-one, where $\pi\colon (X,T)\to (X_{d},T_d)$ be the factor map to its maximal distal factor.
It is easy to see that for a factor map $\pi\colon (X,T)\to (Y,S)$ between minimal systems, it is not almost one-to-one if and only if $\#(\pi^{-1}(y))\geq 2$ for any $y\in Y$.
The following proposition generalizes
the necessary condition of the above result to $n$-tuplewise IP-sensitivity, which is a necessary condition for $n$-tuplewise IP-sensitive system on the fiber of factor map to its maximal distal factor.

\begin{prop}\label{prop:tuplewise-IP-sen-distal-factor}
Let $\pi\colon (X,T)\to (X_{d},T_d)$ be the factor map to its maximal distal factor and $n\geq 2$.
If $(X,T)$ is $n$-tuplewise IP-sensitive then $\#(\pi^{-1}(z))\geq n$ for any $z\in X_d$, where $\#(\,\cdot\,)$ denotes the cardinality of a set.
\end{prop}
\begin{proof}
Pick an $n$-tuplewise IP-sensitive constant $\delta$ for $(X,T)$.
Fix a point $z\in X_d$ and choose a point $x\in \pi^{-1}(z)$.
According to Proposition~\ref{prop:strongly-proximal-IP-n}, for any $m\in\bbn$ there exist
$x_1^{(m)},\dotsc, x_{n}^{(m)}\in B(x,\frac{1}{m})$ and $y_1^{(m)},\dotsc, y_{n}^{(m)}\in X$
such that
$\min_{1\le i<j \le n}d(y_i^{(m)},y_j^{(m)})\geq\delta$ and 	$(x_1^{(m)},\dotsc, x_{n}^{(m)})$ is strongly proximal to $(y_1^{(m)},\dotsc, y_{n}^{(m)})$.
Since $X$ is compact, without loss of generality assume that $y_i^{(m)}\to y_i$, as $m\to\infty$, $i=1,2,\dotsc,n$.
Thus $\min_{1\le i<j \le n}d(y_i,y_j)\geq\delta$.
For any $i=1,2,\dotsc,n$, $(x_i^{(m)},y_i^{(m)})\in P(X,T)\subset R_{\pi}$. Then $(x,y_i)\in R_{\pi}$ since $R_{\pi}$ is closed.
This implies that
$x,y_1,\dotsc,y_n\in \pi^{-1}(z)$.
Since
$y_1,\dotsc, y_{n}$ are pairwise distinct, $\#(\pi^{-1}(z))\geq n$.
\end{proof}

It is interesting to know whether the converse of Proposition \ref{prop:tuplewise-IP-sen-distal-factor} holds for minimal systems.
To be precise, we have the following question.

\begin{ques}
Let $\pi\colon (X,T)\to (X_{d},T_d)$ be the factor map to the maximal distal factor of a minimal system $(X,T)$ and $n\geq 3$.
Assume that $\#(\pi^{-1}(z))\geq n$ for any $z\in X_d$.
Is $(X,T)$ $n$-tuplewise IP-sensitive?
\end{ques}

\subsection{Dichotomy for pairwise IP-sensitivity and pairwise \texorpdfstring{IP$^*$}{IP*}-equicontinuity}
\label{subsec:dichotomy-IP}
In this subsection we study the opposite side of pairwise IP-sensitivity, named pairwise IP$^*$-equicontinuity.

Let $(X,T)$ be a dynamical system.
A point $x$ in $X$ is called \emph{pairwise IP$^*$-equicontinuous} if for any $\eps>0$ there exists a neighbourhood $U$ of $x$ such that for any $y,z\in U$, $\{k\in\bbn\colon d(T^k y, T^k z)<\eps\}$ is an IP$^*$-set. Denote by $\eq^{\ip^*}(X,T)$ the collection of all pairwise IP$^*$-equicontinuous points in $X$.
A dynamical system $(X,T)$ is called \emph{pairwise IP$^*$-equicontinuous} if $\eq^{\ip^*}(X,T)=X$ and
\emph{almost pairwise IP$^*$-equicontinuous}
if $\eq^{\ip^*}(X,T)$ is residual in $X$.	

Since the intersection of two IP$^*$-sets is also an IP$^*$-set, we have the following observations.

\begin{lem}
Let $(X,T)$ be a dynamical system and $x\in X$.
Then $x$ is pairwise IP$^*$-equicontinuous if and only if for any $\eps>0$ there exists a neighbourhood $U$ of $x$ such that for any $y\in U$, $\{k\in\bbn\colon d(T^k x, T^k y)<\eps\}$ is an IP$^*$-set.
\end{lem}

\begin{lem}\label{IP-star-eq-product}
Assume that $(X,T)$ and $(Y,S)$ are two dynamical systems.
\begin{enumerate}
    \item If $x\in X$ and $y\in Y$ are pairwise IP$^*$-equicontinuous
    in $(X,T)$ and $(Y,S)$ respectively,
    then $(x,y)\in X\times Y$ is pairwise IP$^*$-equicontinuous
    in $(X\times Y,T\times S)$.
    \item If both $(X,T)$ and $(Y,S)$ are pairwise IP$^*$-equicontinuous,
    then so is $(X\times Y,T\times S)$.
    \item If both $(X,T)$ and $(Y,S)$ are almost pairwise IP$^*$-equicontinuous,
    then so is $(X\times Y,T\times S)$.
\end{enumerate}
\end{lem}

Now we will characterize the distal property via this weaker form of equicontinuity, first we need some lemmas.

\begin{lem}\label{lem:distal-system-IP-star-eq}
Let $(X,T)$ be a dynamical system. If $(X,T)$ is distal, then it is pairwise IP$^*$-equicontinuous.
\end{lem}
\begin{proof}
Fix a point $x\in X$ and $\eps>0$.
Let $U=B(x,\frac{\eps}{2})$ and $y,z\in U$.
As $y,z$ is distal, $N(y,U)$ and $N(z,U)$ are IP$^*$-sets.
Let $F=N(y,U)\cap N(z,U)$. Then $F$ is also an IP$^*$-set and for any $k\in F$, $d(T^k y,T^k z)<\diam (U)<\eps$.
This implies that $x$ is pairwise IP$^*$-equicontinuous.
As $x$ is arbitrary,
$(X,T)$ is pairwise IP$^*$-equicontinuous.
\end{proof}

\begin{lem}\label{lem:IP-star-eq-point-distal}
Let $(X,T)$ be an IP$^*$-central system and $x\in X$.
If $x$ is pairwise IP$^*$-equicontinuous, then $x$ is distal.
\end{lem}
\begin{proof}
Assume that $x$ is not distal. Then there exists a minimal point $y\in \overline{\orb(x,T)}\setminus\{x\}$ such that $(x,y)$ is proximal.
As $y$ is minimal, $x$ is strongly proximal to $y$.
Let $\delta=\frac{1}{3}d(x,y)$.
By Lemma~\ref{lem:strongly-proximal}, $N(x,B(y,\delta))$ is an IP-set.
As $x$ is pairwise IP$^*$-equicontinuous, there exists a neighborhood $U$ of $x$
such that for any $z\in U$,
$\{k\in\bbn\colon d(T^k x,T^k z)<\delta\}$ is an IP$^*$-set.
However, by Lemma~\ref{lem:supp-IP=subset},
there exists $v\in U\cap B(x,\delta)$ such that $N(v, U\cap B(x,\delta))$
is an IP-subset of $N(x,B(y,\delta))$.
For any $k\in N(v, U\cap B(x,\delta))$,
$d(T^k v,x)<\delta$ and $d(T^k x, y)<\delta$.
Thus $d(T^k v,T^k x)>\delta$ which is a contradiction.
\end{proof}

Combining the above two lemmas, we have the following interesting result.
\begin{thm}
An IP$^*$-central system is IP$^*$-equicontinuous if and only if it is distal.
\end{thm}

As every minimal system is IP$^*$-central, we have the following corollary.

\begin{cor}\label{cor:min-distal-ip-star-eq}
A minimal system is IP$^*$-equicontinuous if and only if it is distal.
\end{cor}

\begin{lem}\label{lem:distal-almost-1-1}
Let $\pi\colon (X,T)\to (Y,S)$ be a factor map.
If $\pi$ is almost one-to-one and $(Y,S)$ is distal,
then $(X,T)$ is almost pairwise IP$^*$-equicontinuous.
\end{lem}
\begin{proof}
Let $X_0=\{x\in X\colon \pi^{-1}(\pi(x))=\{x\}\}$.
As $\pi$ is almost one-to-one, $X_0$ is residual in $X$.
It is sufficient to show that $x$ is pairwise IP$^*$-equicontinuous.
Fix $\eps>0$. There exists a neighborhood $V$ of $\pi(x)$
such that $\pi^{-1}(V)\subset B(x,\frac{\eps}{2})$.
Pick a neighborhood $U$ of $x$ with $\pi(U)\subset V$.
For any $u,v\in U$, $\pi(u),\pi(v)\in V$.
As $\pi(u)$ and $\pi(v)$ are distal,
\[F:=\{k\in\bbn\colon S^k \pi(u), S^k \pi(v)\in V\}
\]
is an IP$^*$-set.
For any $k\in F$, $T^k u,T^k v\in \pi^{-1}(V) \subset B(x,\frac{\eps}{2})$,
and then $d(T^k u,T^k v)<\eps$.
This implies that $x$ is pairwise IP$^*$-equicontinuous.
\end{proof}

We have the following dichotomy result for minimal systems.

\begin{thm} \label{thm:minimal-dichotomy-IP}
Let $(X,T)$ be a minimal system. Then $(X,T)$ is either pairwise IP-sensitive or almost pairwise IP$^*$-equicontinuous.
\end{thm}
\begin{proof}
Let $\pi\colon (X,T)\to (X_d,T_d)$ be the factor map to its maximal distal factor.
If $\pi$ is almost one-to-one, then by Lemma~\ref{lem:distal-almost-1-1} $(X,T)$ is
almost pairwise IP$^*$-equicontinuous.
If $\pi$ is not almost one-to-one, then by
\cite[Theorem~C]{YY18} $(X,T)$ is
pairwise IP-sensitive.
\end{proof}

\begin{cor} \label{cor:minimal-pairwise-IP-star-eq}
Let $(X,T)$ be a minimal system and $\pi\colon (X,T)\to (X_{d},T_d)$ be the factor map to its maximal distal factor. Then the following statements are equivalent:
\begin{enumerate}
\item $(X,T)$ is almost pairwise IP$^*$-equicontinuous;
\item $(X,T)$ is point-distal and $P(X,T)$ is closed;
\item $\pi$ is almost one-to-one.
\end{enumerate}
\end{cor}
\begin{proof}
(1)$\Leftrightarrow$(3) follows from
Theorem \ref{thm:minimal-dichotomy-IP} and \cite[Theorem~C]{YY18};
(2)$\Leftrightarrow$(3) follows from \cite[Page 600]{Vries1993}.
\end{proof}

\begin{prop}\label{prop:almost-1-1-IP-star-eq}
Let $\pi\colon (X,T)\to (Y,S)$ be a factor map between minimal systems.
\begin{enumerate}
    \item If $(X,T)$ is almost pairwise IP$^*$-equicontinuous then  so is $(Y,S)$.
    \item If $\pi$ is almost one-to-one, then $(X,T)$ is almost pairwise IP$^*$-equicontinuous if and only if so is $(Y,S)$.
\end{enumerate}
\end{prop}
\begin{proof}(1) Note that $\pi$ is semi-open, as it is a factor map between minimal systems.
	By Lemma~\ref{lem:IP-sen-semi-open}	
	and Theorem~\ref{thm:minimal-dichotomy-IP}, $(Y,S)$ is almost pairwise IP$^*$-equicontinuous.

	(2) Now assume that $(Y,S)$ is almost pairwise IP$^*$-equicontinuous.
	Let $\theta\colon (Y,S)\to (Y_d,S_d)$ be the factor map to its maximal distal factor.
	By Corollary~\ref{cor:minimal-pairwise-IP-star-eq}  $\theta$ is almost one-to-one.
	As $\pi$ is almost one-to-one, $\theta\circ \pi\colon (X,T)\to (Y_d,S_d)$ is also almost one-to-one.
	By Lemma~\ref{lem:distal-almost-1-1} $(X,T)$ is almost pairwise IP$^*$-equicontinuous.
\end{proof}

\subsection{Maximal almost pairwise $IP^*$-equicontinuous factor}
We know that every dynamical system has a maximal equicontinuous factor. Now we will show that every minimal system $(X,T)$ admits a maximal almost pairwise IP$^*$-equicontinuous factor.

\begin{thm}\label{thm:max-ip-*}
Each minimal system $(X,T)$ admits a maximal almost pairwise $IP^*$-equicontinuous factor.
\end{thm}
\begin{proof} Let $\psi: (X,T)\rightarrow (X_d,T_d)$ be a factor map to its maximal distal factor. 
	Let $\mathcal{A}$ consist all the  closed $T\times T$  invariant equivalence relation $R$ on  $X$ such that $(X/R, T_R)$ is an almost one to one extension of $(X_d,T_d)$. Since $R_\psi \in \mathcal{A}$,  $\mathcal{A}$ is not empty.
	Let $R^*=\cap\{R:R\in \mathcal{A}\}$, then $R^*$ is closed $T\times T$  invariant equivalence relation. Since $X$ is a compact metric space,
	then there exists a countable $R_i\in \mathcal{A}$ such that $R^*=\cap_{i=1}^\infty R_i$.
	Since $(X/R^*,T_{R^*})$ is an inverse limit of $(X/R_i,T_{R_i})$ with $(X/R_i,T_{R_i})$  an almost one to one extension of $(X_d,T_d)$, then
	$(X/R^*,T_{R^*})$ is an almost one to one extension of $(X_d,T_d)$ (see\cite[Page 592]{Vries1993}),
	$R^*\in \mathcal{A}$. Thus $(X/R^*,T_{R^*})$ is the maximal almost one to one extension of $(X_d,T_d)$.
By Proposition
	~\ref{prop:almost-1-1-IP-star-eq},
	$(X/R^*,T_{R^*})$ is almost pairwise $IP^*$-equicontinuous.
	
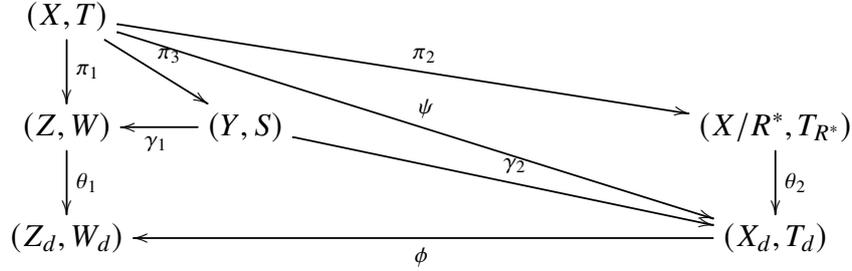
\[\xymatrix{
		(X,T) \ar[drrrrrr]^{\pi_2} \ar[dr]^{\pi_3} \ar[d]^{\pi_1}  \ar[drrdrrrr]^{\psi}         \\
		(Z,W) \ar[d]^{\theta_1}
		&(Y,S)\ar[l]^{\gamma_1} \ar[drrrrr]^{\gamma_2}
		&&&&& (X/R^*,T_{R^*})\ar[d]^{\theta_2} \\
		(Z_d,W_d)
		&&&&&&(X_d,T_d)  \ar[llllll]^{\phi}
	}\]
	We next show that $(X/R^*,T_{R^*})$ is the maximal almost pairwise $IP^*$-equicontinuous factor of $(X,T)$.
	
	Let $(Z,W)$ be an almost pairwise $IP^*$ equicontinuous factor of $(X,T)$. We are sufficient to show that $(Z,W)$ is a factor
	of $(X/R^*,T_{R^*})$.
	Let $\pi_1:(X,T)\rightarrow (Z,W)$
	and $\theta_1:(Z,W)\rightarrow (Z_d,W_d)$
	be factor maps.
	By Corollary~\ref{cor:minimal-pairwise-IP-star-eq}, $\theta_1$ is almost one to one.
	Denote $ R_{\pi_3}=R_{\pi_1}\cap R_\psi$ and let $(Y,S)=(X/R_{\pi_3},T_{R_{\pi_3}})$.
	It is clear that we have the following factor maps:
	$\gamma_1:(Y,S)\rightarrow (Z,W)$
	and
	$\gamma_2:(Y,S)\rightarrow (X_d,T_d)$. If $\gamma_2$ is almost one to one, then $(Y,S)$ is a factor of $(X/R^*,T_{R^*})$ and $(Z,W)$ is a factor
	of $(X/R^*,T_{R^*})$.
	
Now we show that $\gamma_2$ is almost one to one. 
	Since $\theta_1$ is almost one to one,
	there exists $x\in X$ such that
	$R_{\theta_1\circ \pi_1}(x)=R_{\pi_1}(x)$. It easy to see that
	$$R_\psi(x)\subset R_{\phi \circ \psi}(x) = R_{{\theta_1}\circ \pi_1}(x)=R_{\pi_1}(x).$$
	Thus $R_{\gamma_2 \circ \pi_3}(x)=R_\psi(x)=R_\psi(x)\cap R_{\pi_1}(x)=(R_{\pi_1}\cap R_\psi)(x)=R_{\pi_3}(x)$.
	Therefore $\gamma_2$ is almost one to one.
\end{proof}

We end this section by leaving the following question:

\begin{ques}
For a minimal system $(X,T)$, how to describe the $T\times T$-invariant closed equivalence relation on $X$ determined by the maximal almost pairwise IP$^*$-equicontinuous factor of $(X,T)$?
\end{ques}

\section{\texorpdfstring{$N$}{N}-tuplewise thick sensitivity}

Let $(X,T)$ be a dynamical system and $n\geq 2$.
We say that $(X,T)$ is \emph{$n$-tuplewise thickly sensitive}
if there exists a constant $\delta>0$ with the property that for any opene subset $U$ of $X$ there exist $x_1,\dots,x_n\in U$
such that
\[
  \Bigl\{ k\in\bbn\colon \min_{1\le i<j\le n}d(T^k x_i,T^k x_j) >\delta\Bigr \}
\]
is a thick set. The constant $\delta$ is called an \emph{$n$-tuplewise thickly sensitive constant} for $(X,T)$.
For simplicity, $2$-tuplewise thick sensitivity will be called \emph{pairwise thick sensitivity}.

Since each thick set is an IP-set, every $n$-tuplewise thickly sensitive system is also $n$-tuplewise IP-sensitive.

\subsection{Properties of \texorpdfstring{$n$}{N}-tuplewise thick sensitivity}
For transitive systems, we have the following characterization for $n$-tuplewise thickly sensitivity.

\begin{lem}\label{lem:transitive-thick-sen}
Let $(X,T)$ be a transitive system and $x\in \trans(X,T)$.
If there exists a constant $\delta>0$ with the property that
for every neighborhood $U$ of $x$ there exist $x_1,\dotsc,x_n\in U$ such that
\[
 \Bigl \{ k\in\bbn\colon \min_{1\le i<j\le n}d(T^k x_i,T^k x_j)>\delta\Bigr \}
\]
is a thick set, then $(X,T)$ is $n$-tuplewise thickly sensitive.
\end{lem}
\begin{proof}
For any opene subset $V$ of $X$, there exists $m\in \bbn$ such that $T^m x\in V$.
Pick a neighborhood $U$ of $x$ such that $T^m U\subset V$.
There exist $x_1,\dotsc,x_n\in U$ such that
\[
 \Bigl\{ k\in\bbn\colon \min_{1\le i<j\le n}d(T^k x_i,T^k x_j)>\delta\Bigr\}
\]
is  thick.
Then $T^m x_i\in V$ for $i=1,\dotsc,n$ and
\[
 \Bigl\{ k\in\bbn\colon \min_{1\le i<j\le n}d(T^k(T^mx_i),T^k(T^mx_j))>\delta\Bigr\}
\]
is also thick, since a translation of a thick set is also thick.
This implies that $(X,T)$ is $n$-tuplewise thickly sensitive.
\end{proof}

Now we give some sufficient and necessary conditions for $n$-tuplewise thick sensitivity.

\begin{prop}\label{prop:n-thick-proximal-1}
Let $(X,T)$ be a dynamical system and $n\geq 2$. Then the following statements are equivalent:
\begin{enumerate}
    \item $(X,T)$ is $n$-tuplewise thickly sensitive;
	\item there is a constant $\delta>0$ with the property that for each opene subset $U$ of $X$,
	there exist $x_1,\dotsc,x_n \in U$	and $y_1,\dotsc,y_n\in X$
	such that $(x_1,\dotsc,x_n)$ is proximal to $(y_1,\dotsc,y_n)$ in the $n$-fold product system $(X^n,T^{(n)})$
	and
	\[
	\liminf_{m\to\infty} \min_{1\le i < j\le n}d(T^m y_i,T^my_j)\ge\delta;
	\]
	\item there is a constant $\delta>0$ with the property that for each opene subset $U$ of $X$,
	there exist $x_1,\dotsc,x_n\in U$ and $y_1,\dotsc,y_n\in X$
	such that $(x_1,\dotsc,x_n)$ is proximal to $(y_1,\dotsc,y_n)$ which is a minimal point in the $n$-fold product system   $(X^n,T^{(n)})$ and
	\[
	\liminf_{m\to\infty} \min_{1\le i < j\le n}d(T^m y_i,T^my_j)\ge\delta.
	\]
	\end{enumerate}
\end{prop}
\begin{proof}
	$(2)\Rightarrow (1)$
	Let $\delta$ be the constant with the desired property.
	Fix an opene subset $U$ of $X$.
	There exist $x_1,\dotsc,x_n \in U$	and $y_1,\dotsc,y_n\in X$
	such that $(x_1,\dotsc,x_n)$ is proximal to $(y_1,\dotsc,y_n)$
	and
	\[
	\liminf_{m\to\infty} \min_{1\le i < j\le n}d(T^m y_i,T^my_j)\ge\delta;
	\]
	Let $M=\omega((y_1,y_2,\dotsc,y_n),T^{(n)})$.
	Then for any $(z_1,\dotsc,z_n)\in M$,
	$\min_{1\leq i<j\leq n}d(z_i,z_j)\geq \delta$.
	Choose a neighbourhood $W$ of $M$ such that
	for any $(z_1^{\prime},\dotsc,z_n^{\prime})\in W$,
	$\min_{1\leq i<j\leq n}d(z_i^{\prime},z_j^{\prime})>\frac{\delta}{2}$.
	For any $L>0$, since $M$ is $T^{(n)}$-invariant, $\bigcap_{i=0}^L (T^{(n)})^{-i}W$ is also a neighbourhood of $M$.
	Since $(x_1,\dotsc,x_n)$ and $(y_1,\dotsc,y_n)$ are proximal,
	there exists an $m\in \bbn$ such that
	$(T^{(n)})^m(x_1,\dotsc,x_n)\in\bigcap_{i=0}^L (T^{(n)})^{-i}W$, that is $(T^{(n)})^{m+i}(x_1,\dotsc,x_n)\in W$, $i=0,1,\dotsc, L$. According to the construction of $W$,
	\[
	\{m,m+1,\dotsc,m+L\}\subset
	\Bigl\{k\in\bbn\colon \min_{1\le i<j\le n}d(T^kx_i,T^kx_j)>\tfrac{\delta}{2}\Bigr\}.
	\]
	Then $\frac{\delta}{2}$ is an $n$-tuplewise thickly sensitive constant for $(X,T)$.
	
$(3)\Rightarrow (2)$ is obvious.
		
$(1)\Rightarrow (3)$
	Assume that $\delta$ is an $n$-tuplewise thickly sensitive constant for $(X,T)$. Fix an opene subset $U$ of $X$.
	There exist $x_1,\dotsc, x_n\in U$ such that
\[
  F:=\Bigl\{ k\in\bbn\colon \min_{1\le i<j\le n}d(T^k x_i,T^k x_j) >\delta\Bigr \}
\]
    is a thick set.
Then there exists  a sequence $\{q_m\}_{m=1}^{\infty}$ in $\bbn$
	such that
	\[
	\bigcup_{m=1}^{\infty}(q_m+\{0,1,\dotsc,m\})\subset F.
	\]
Without loss of generality, assume that
$T^{q_m}x_i\to z_i\in X$ as $m\to\infty$ for $i=1,\dotsc,n$.
Then $\orb((z_1,\dotsc,z_n),T^{(n)})\subset \overline{\orb((x_1,\dotsc,x_n),T^{(n)})}$.
By Lemma \ref{lem:orbit-minimal-proximal}, there exists  a minimal point $(y_1,y_2,\dotsc,y_n) \in \overline{\orb((z_1,\dotsc,z_n),T^{(n)})}$ such that $(x_1,x_2,\dotsc, x_n)$ and $(y_1,y_2,\dotsc,y_n)$ are proximal.
For any $p\in\bbn$, $q_m+p \in F$ for all $m\geq p$.
So we have
\[
\min_{1\leq i<j\leq n} d(T^p z_i,T^p z_j)
= \lim_{m\to\infty}\min_{1\leq i<j\leq n} d(T^{q_m+p} x_i,T^{q_m+p} x_j)\geq \delta
\]
and then
\[
\inf_{p\in\bbn}\min_{1\leq i<j\leq n} d(T^p y_i,T^p y_j)\geq \min_{1\leq i<j\leq n} d(T^p z_i,T^p z_j)\geq \delta.
\]
Furthermore, one has
\[
\liminf_{p\to\infty} \min_{1\le i < j\le n}d(T^p y_i,T^p y_j)
\geq \inf_{p\in\bbn}\min_{1\leq i<j\leq n} d(T^p y_i,T^p y_j)\geq \delta.
\]
This ends the proof.
\end{proof}

Now we show a necessary and sufficient condition for a non-trivial weakly mixing system to be $n$-tuplewise thickly sensitive by using Proposition~\ref{prop:n-thick-proximal-1}.

\begin{thm}\label{thm:w-mixing-thick-AP}
Assume that $(X,T)$ is a non-trivial weakly mixing system and $n\geq 2$.
Then $(X,T)$ is $n$-tuplewise thickly sensitive if and only if
it has at least $n$ minimal points.
\end{thm}
\begin{proof}
$(\Rightarrow)$
As $(X,T)$ is $n$-tuplewise thickly sensitive,
by Proposition~\ref{prop:n-thick-proximal-1}
there exists a minimal point
$(y_1,\dotsc,y_n)\in X^{(n)}$ with
\[
\liminf_{m\to\infty} \min_{1\le i < j\le n}d(T^m y_i,T^my_j)>0.
\]
Thus $y_1,\dotsc,y_n$ are pairwise distinct minimal points in $X$.
		
$(\Leftarrow)$
Assume that there exist $n$ pairwise distinct minimal points $z_1, \dotsc, z_n$ in $X$.
Let $Z= \overline{\orb(z_1,T)}\times \dotsb\times \overline{\orb(z_n,T)}$.
Then $Z$ is $T^{(n)}$-invariant and has a dense set of minimal points.
Pick a minimal point $(w_1,\dotsc,w_n)\in Z$ with pairwise distinct coordinates and let
\[
\delta= \inf_{k\in \bbn}\min_{1\le i<j\le n}d(T^kw_i,T^k w_j).
\]
Then $\delta>0$ and for any $(y_1, \dotsc, y_n)\in \overline{\orb((w_1,\dotsc,w_n),T^{(n)})}$,
\[
\inf_{k\in \bbn}\min_{1\le i<j\le n}d(T^ky_i,T^k y_j)
=\inf_{k\in \bbn}\min_{1\le i<j\le n}d(T^kw_i,T^k w_j)=\delta>0.
\]
Fix an opene subset $U$ of $X$.
By Lemma~\ref{lem:weakly-mixing-n}, $(X^n, T^{(n)})$ is  weakly mixing.  Pick a transitive points $(x_1,\dotsc,x_n)\in U^n$.
By Lemma~\ref{lem:orbit-minimal-proximal},
there exists a point $(y_1, \dotsc, y_n)\in \overline{\orb((w_1,\dotsc,w_n),T^{(n)})}$
such that $(x_1,\dotsc,x_n)$ is proximal to $(y_1,\dotsc,y_n)$.
By Proposition~\ref{prop:n-thick-proximal-1}, $(X,T)$ is $n$-tuplewise thickly sensitive.
\end{proof}
	
\begin{cor}
If $(X,T)$ is a non-trivial  minimal weakly mixing system, then $(X,T)$ is $n$-tuplewise thickly sensitive for any $n\geq 2$.
\end{cor}
	
\begin{rem}\label{rem:thick-exam}
For every $n\geq 2$, by \cite[Proposition 6.3]{HKKPZ18} there exists a non-trivial strongly mixing system with exact $n$ minimal points.
Then by Theorem~\ref{thm:w-mixing-thick-AP}, this strongly mixing system is $n$-tuplewise thickly sensitive, but not
$(n+1)$-tuplewise thickly sensitive.
\end{rem}

\begin{rem}\label{rem:thick-ip-exam}
There exists a non-trivial strongly mixing system $(X,T)$ with a unique minimal point (see e.g. \cite[Main Theorem]{HZ} or \cite[Proposition 6.3]{HKKPZ18}).
Then by Theorem~\ref{thm:w-mixing-thick-AP},
$(X,T)$ is not pairwise thickly sensitive.
However, by Theorem~\ref{thm:weakly-mixing-IP-sen} $(X,T)$ is $n$-tuplewise IP-sensitive for all $n\geq 2$.
\end{rem}

\subsection{\texorpdfstring{$N$}{N}-tuplewise thick sensitivity among factor maps}
Similar to the Lemma~\ref{lem:IP-sen-semi-open},
$n$-tuplewise thick sensitivity can be also lifted by semi-open factor map. We leave the routine proof to the reader.

\begin{lem}\label{lem:semi-open-thick}
Let $\pi\colon (X,T)\to (Y,S)$ be a factor map and $n\geq 2$.
If $(Y,S)$ is $n$-tuplewise thickly sensitive and $\pi$ is semi-open, then $(X,T)$ is also $n$-tuplewise thickly sensitive.
\end{lem}

The following result shows that $n$-tuplewise thick sensitivity is preserved by proximal factor maps, which generalizes the result of \cite[Proposition 5.1]{YY18}.

\begin{prop}\label{prop:pi-proximal-n-thick}
Let $\pi\colon (X,T)\to (Y,S)$ be a factor map and $n\geq 2$.
If $(X,T)$ is $n$-tuplewise thickly sensitive and $\pi$ is proximal,
then $(Y,S)$ is also $n$-tuplewise thickly sensitive.	
\end{prop}
\begin{proof}
We prove this result by contradiction.
Assume that $(Y,S)$ is not $n$-tuplewise thickly sensitive.
Then for every $m\in\bbn$ there exists an opene subset $U_m$ of $Y$ such that for any  $y_{1,m},\dotsc,y_{n,m}\in U_m$,
\[
\Bigl\{ k\in\bbn\colon \min_{1\leq i<j\leq n}d(S^k y_{i,m},S^k y_{j,m})<\tfrac{1}{m}
\Bigr\}
\]
is syndetic.
Pick an $n$-tuplewise thickly sensitive constant $\delta$ for $(X,T)$.
For any $m\in\bbn$, there exist $x_{1,m},\dotsc,x_{n,m}\in \pi^{-1}(U_m)$ such that
\[
\Bigl\{ k\in\bbn\colon \min_{1\leq i<j\leq n}d(T^k x_{i,m}, T^k x_{j,m})>\delta
\Bigr\}
\]
is thick.
Then for each $m\in\bbn$, there exist $k_m\in\bbn$ and $1\leq i_m<j_m\leq n$ such that
\[
d(T^{k_m+q} x_{i_m,m}, T^{k_m+q}x_{j_m,m})>\delta\]
for $q=0,1,\dotsc,m$ and
\[
d\bigl(S^{k_m}(\pi( x_{i_m,m})),S^{k_m}(\pi( x_{j_m,m}))\bigr)<\tfrac{1}{m}.
\]
Without loss of generality, assume that
$T^{k_m}x_{i_m,m} \to x_{i_0}$ and
$T^{k_m}x_{j_m,m} \to x_{j_0}$ as $m\to\infty$.
Then \[
d(\pi(x_{i_0}),\pi(x_{j_0}))
=\lim_{m\to\infty}d\bigl(S^{k_m}(\pi( x_{i_m,m})),S^{k_m}(\pi( x_{j_m,m}))=0,
\]
that is $\pi(x_{i_0})=\pi(x_{j_0})$.
For any $q\geq 0$, one has
\[
d(T^q x_{i_0},T^q x_{j_0})
=\lim_{m\to\infty}d(T^{k_m+q} x_{i_m,m}, T^{k_m+q}x_{j_m,m})
\geq \delta,
\]
which contradicts to the proximality of $\pi$.
\end{proof}

Since every factor between minimal systems are semi-open, according to  Lemma~\ref{lem:semi-open-thick} and Proposition~\ref{prop:pi-proximal-n-thick} we have the following corollary.

\begin{cor}\label{cor:proximal-factor-thick-sen}
Let $\pi\colon (X,T)\to (Y,S)$ be a factor map between minimal systems. Assume that $\pi$ is proximal and $n\geq 2$.
Then $(X,T)$ is $n$-tuplewise thickly sensitive if and only if so is $(Y,S)$.
\end{cor}

It is shown in \cite[Theorem~D]{YY18} that
a minimal system $(X,T)$ is pairwise thickly sensitive if and only if
$\pi$ is not proximal, where $\pi\colon (X,T)\to (X_{d},T_d)$ be the factor map to its maximal distal factor.
It is easy to see that for a factor map $\pi\colon (X,T)\to (X_d,T_d)$ between minimal systems, it is not  proximal if and only if for any $y\in Y$ there exist distinct $x_1,x_2\in\pi^{-1}(y)$ such that $(x_1,x_2)$ is a minimal point in $(X^2,T^{(2)})$.
Now we generalizes
the necessary condition of the above result to $n$-tuplewise thick sensitivity for systems which are not necessary to be minimal.

\begin{prop}\label{prop:thick-sen-minimal-point}
Let $(X,T)$ be a dynamical system and $\pi:(X,T)\to (X_{d},T_d)$ be the factor map to its maximal distal factor.
Assume that $n\geq 2$.
If $(X,T)$ is $n$-tuplewise thickly sensitive then there exists $z\in X_d$ and pairwise distinct $z_1,\dotsc,z_n\in \pi^{-1}(z)$ such that
$(z_1,\dotsc,z_n)$ is a minimal point in $(X^n,T^{(n)})$.
\end{prop}
	\begin{proof}	
	For each $k\in\bbn$, there is $\tau_k>0$ such that
	if $w_1,w_2\in X$ with $d(w_1,w_2)<2\tau_k$ then $d(\pi(w_1),\pi(w_2))<\frac{1}{k}$.
	Pick $x\in X$ and put $U_k=B(x,\tau_k)$.
	By the assumption $(X,T)$ is $n$-tuplewise thickly sensitive, by Proposition \ref{prop:n-thick-proximal-1}
	there are $(x_1^k,\dotsc, x_n^k)\in (U_k)^n$ and $(y_1^k,\dotsc, y_n^k)\in X^n$ with $(x_1^k,\dotsc,x_n^k)$  proximal to $(y_1^k,\dotsc,y_n^k)$ and
	\[
	\inf_{m\in \bbn}\min_{1\le i<j\le n}d(T^my_i^k,T^my_j^k)\ge\delta.
	\]
	Without loss of generality we assume that $y_i^k\to y_i$, $i=1,2,\cdots,n$, when $k\to \infty$.
	It is clear that
	$d(T^my_i,T^my_j)\ge \delta$, $1\leq i\neq j\leq n$,
	for each $m\in \bbn$.
	Now we show that $\pi(y_i)=\pi(y_j)$,
	for $1\leq i\neq j\leq n$.
	Since $x_1^k,\dotsc, x_n^k\in U_k$, $d(\pi(x_i^k),\pi(x_j^k))<\frac{1}{k}$.
	Since $(x_1^k,\cdots,x_n^k)$ is proximal to $(y_1^k,\dotsc,y_n^k)$ and $X_{d}$ is a distal system, $\pi(x_i^k)=\pi(y_i^k)$.
	Thus $d(\pi(y_i^k),\pi(y_j^k))<\frac{1}{k}$, $d(\pi(y_i),\pi(y_j))=0$ and $\pi(y_i)=\pi(y_j)$.
	Choose a minimal point $(z_1,\dotsc,z_n)\in \overline{\orb{((y_1,\dotsc,y_n),T^{(n)})}}\subset R_n^\pi$.
	Then there exists $z\in X_d$ such that $z_1,\dotsc,z_n\in \pi^{-1}(z)$.
\end{proof}

It is interesting to know whether the converse of Proposition~\ref{prop:thick-sen-minimal-point} holds for minimal systems.
To be precise, we have the following question.

\begin{ques}
Let $\pi\colon (X,T)\to (X_{d},T_d)$ be the factor map to the maximal distal factor of a minimal system $(X,T)$ and $n\geq 3$.
Assume that for every $z\in X_d$ there exist pairwise distinct $x_1,\dotsc,x_n\in \pi^{-1}(z)$ such that
$(x_1,\dotsc,x_n)$ is a minimal point in $(X^n,T^{(n)})$.
Is $(X,T)$ $n$-tuplewise thickly sensitive?
\end{ques}

\subsection{Dichotomy for pairwise thick sensitivity and pairwise syndetic equicontinuity}
In this subsection we study the opposite side of pairwise thick sensitivity, named pairwise syndetic equicontinuity.

Let $(X,T)$ be a dynamical system.
A point $x\in X$ is  called \emph{pairwise syndetically equicontinuous} if for any $\eps>0$ there exists a neighborhood $U$ of $x$ such that for any $y,z\in U$,
$\{k\in\bbn\colon d(T^k y,T^k z)<\eps\}$ is a syndetic set.
If each point $x\in X$ is pairwise syndetically equicontinuous, then we say that $(X,T)$ is pairwise syndetically equicontinuous.

As each IP$^*$-set is syndetic, every pairwise IP$^*$-equicontinuous system is also pairwise syndetically equicontinuous.
In particular, by Lemma~\ref{lem:distal-system-IP-star-eq}  every distal system is pairwise syndetically equicontinuous.

Denote by $\eq^{\syn}(X,T)$ the collection of all pairwise syndetically equicontinuous points in $X$.
A dynamical system $(X,T)$ is called \emph{almost pairwise syndetically equicontinuous} if $\eq^{\syn}(X,T)$ is residual in $X$.

In the following we will show the dichotomy result, since
the collections of syndetic sets and thick sets are translation invariant, we use the standard techniques.

\begin{lem}\label{lem:syn-equi-points}
Let $(X,T)$ be a dynamical system.
Then $\eq^{\syn}(X,T)$ is a $G_\delta$ subset of $X$ and $T^{-1}(\eq^{\syn}(X,T))\subset \eq^{\syn}(X,T)$.
\end{lem}
\begin{proof}
For each $m\in\bbn$, denote by $\eq^{syn}_m(X,T)$ the collection of all points $x$ in $X$ with the property that there exists a $\delta>0$ such that for any $y,z\in B(x,\delta)$,
$\bigl \{k\in\bbn\colon  d(T^k y,T^k z)<\frac{1}{m}\bigr\}$ is a syndetic set.
Clearly, $\eq^{syn}_m(X,T)$  is an open subset of $X$.
As every syndetic system is translation invariant, $T^{-1}(\eq^{syn}_m(X,T))\subset \eq^{syn}_m(X,T))$.
Then $\eq^{syn}(X,T)$ is a   $G_\delta$ subset of $X$ and  $T^{-1}(\eq^{syn}(X,T))\subset \eq^{syn}(X,T)$,
because $\eq^{syn}(X,T)=\bigcap_{m=1}^{\infty} \eq^{syn}_m(X,T)$.
\end{proof}

We have the following dichotomy result for transitive systems.
\begin{thm}	\label{thm:trans-dichotomy-thick}
Let $(X,T)$ be a transitive system. Then $(X,T)$ is either pairwise thickly sensitive or almost pairwise syndetically equicontinuous.
\end{thm}
\begin{proof}
	First assume that $\eq^{syn}(X,T)$ is not empty.
	For each $m\in\bbn$, let $\eq^{syn}_m(X,T)$ as in the proof of Lemma \ref{lem:syn-equi-points}.
	As $\eq^{syn}(X,T)$ is not empty, $\eq^{syn}_m(X,T)$ an opene subset of $X$.
	For any transitive point $x\in X$, there is $k\in\bbn$ such that $T^k x\in \eq^{syn}_m(X,T)$.
	Since $T^{-1}(\eq^{syn}_m(X,T))\subset \eq^{syn}_m(X,T)$,  $x\in \eq^{syn}_m(X,T)$. By the arbitrariness  of $m$, $x\in \eq^{syn}(X,T)$.
	This shows that $\trans(X,T)\subset \eq^{syn}(X,T)$.
	Therefore, $(X,T)$ is almost pairwise syndetically equicontinuous.
	
	Now assume that $\eq^{syn}(X,T)$ is empty.
	Fix a transitive point $x\in X$.
	As $x$ is not pairwise syndetic equicontinuous, there exists
	a $\delta>0$ with the property that for any neighborhood $U$ of $x$
	there exist $y,z\in U$ such that $\{k\in\bbn\colon d(T^k y, T^k z)>\delta\}$ is a thick set.
	Then by Lemma \ref{lem:transitive-thick-sen} $(X,T)$ is pairwise thickly sensitive.
\end{proof}

For a minimal system $(X,T)$, we have $\trans(X,T)=X$.
By the proof of Theorem~\ref{thm:trans-dichotomy-thick}, in this case if
$\eq^{syn}(X,T)$ is not empty then $\eq^{syn}(X,T)=X$.
So we have the following corollary.
\begin{cor}\label{cor:minimal-dichotomy-thick}
Let $(X,T)$ be a minimal system. Then $(X,T)$ is either pairwise thickly sensitive or pairwise syndetically equicontinuous.
\end{cor}

Combining \cite[Theorem~D]{YY18} and Corollary~\ref{cor:minimal-dichotomy-thick}, we have the following result.

\begin{cor}\label{cor:distal-syndetic-proximal}
Let $(X,T)$ be a minimal system and $\pi\colon(X,T)\to (X_d,T_d)$ be the factor map to its maximal distal factor.
Then the following statements are equivalent:
\begin{enumerate}
\item $(X,T)$ is pairwise syndetically equicontinuous;
\item $\pi$ is proximal;
\item $P(X,T)$ is closed.
\end{enumerate}
\end{cor}
	
Combining Lemma \ref{lem:semi-open-thick},  Corollaries \ref{cor:proximal-factor-thick-sen} and \ref{cor:minimal-dichotomy-thick}, we have the following result.
	
\begin{cor} \label{cor:factor-syn-eq}
Let $\pi\colon (X,T)\to (Y,S)$ be a factor map between minimal systems.
\begin{enumerate}
    \item If $(X,T)$ is pairwise syndetically equicontinuous then  so is $(Y,S)$.
    \item If $\pi$ is proximal, then $(X,T)$ is pairwise syndetically equicontinuous if and only if so is $(Y,S)$.
\end{enumerate}
\end{cor}

\subsection{Maximal almost pairwise syndetic-equicontinuous factor}
Similar to Theorem~\ref{thm:max-ip-*},
we will show that every minimal system $(X,T)$ admits a maximal pairwise syndetic equicontinuous factor.
\begin{thm}\label{thm:max-syndetic}
Every minimal system $(X,T)$ admits a
	maximal pairwise syndetically equicontinuous factor.
\end{thm}
\begin{proof} Let $\psi: (X,T)\rightarrow (X_d,T_d)$ be a factor map to its maximal distal factor.
	Let $\mathcal{A}$ consist all the  closed $T\times T$  invariant equivalence relation $R$ on  $X$ such that $(X/R, T_R)$ is a proximal extension of $(X_d,T_d)$. Since $R_\psi \in \mathcal{A}$,  $\mathcal{A}$ is not empty.
	Let $R^*=\cap\{R:R\in \mathcal{A}\}$, then $R^*$ is closed $T\times T$  invariant equivalence relation. Since $X$ is a compact metric space,
	then there exists a countable $R_i\in \mathcal{A}$ such that $R^*=\cap_{i=1}^\infty R_i$.
	Since $(X/R^*,T_{R^*})$ is an inverse limit of $(X/R_i,T_{R_i})$ with $(X/R_i,T_{R_i})$  a proximal extension of $(X_d,T_d)$, then
	$(X/R^*,T_{R^*})$ is a proximal extension of $(X_d,T_d)$,
	$R^*\in \mathcal{A}$ (see\cite[Page 416]{Vries1993}). Thus $(X/R^*,T_{R^*})$ is the maximal proximal extension of $(X_d,T_d)$.
By Corollary
	~\ref{cor:distal-syndetic-proximal},
	$(X/R^*,T_{R^*})$ is pairwise syndetically equicontinuous.
	
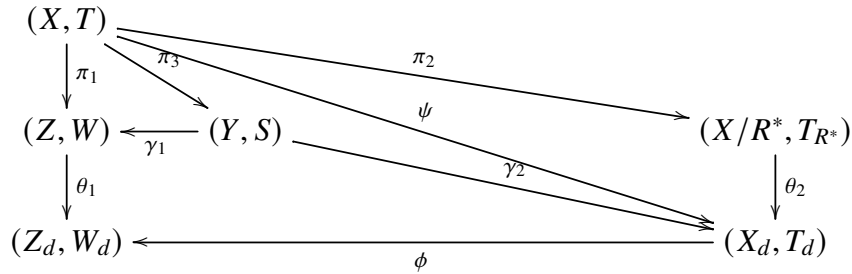
\[\xymatrix{
	(X,T) \ar[drrrrrr]^{\pi_2} \ar[dr]^{\pi_3} \ar[d]^{\pi_1}  \ar[drrdrrrr]^{\psi}         \\
	(Z,W) \ar[d]^{\theta_1}
	&(Y,S)\ar[l]^{\gamma_1} \ar[drrrrr]^{\gamma_2}
	&&&&& (X/R^*,T_{R^*})\ar[d]^{\theta_2} \\
	(Z_d,W_d)
	&&&&&&(X_d,T_d)  \ar[llllll]^{\phi}
}\]
	We next show that $(X/R^*,T_{R^*})$ is the maximal pairwise syndetically equicontinuous factor of $(X,T)$.
	
	Let $(Z,W)$ be a pairwise syndetically equicontinuous factor of $(X,T)$. We are sufficient to show that $(Z,W)$ is a factor
	of $(X/R^*,T_{R^*})$.
	Let $\pi_1:(X,T)\rightarrow (Z,W)$
	and $\theta_1:(Z,W)\rightarrow (Z_d,W_d)$
	be factor maps.
	By  Corollary
	~\ref{cor:distal-syndetic-proximal}, $\theta_1$ is proximal.
	Denote $R_{\pi_3}=R_{\pi_1}\cap R_\psi$ and let $(Y,S)=(X/R_{\pi_3},T_{R_{\pi_3}})$.
	It is clear that we have the following factor maps:
	$\gamma_1:(Y,S)\rightarrow (Z,W)$
	and
	$\gamma_2:(Y,S)\rightarrow (X_d,T_d)$. If $\gamma_2$ is proximal, then $(Y,S)$ is a factor of $(X/R^*,T_{R^*})$ and $(Z,W)$ is a factor
	of $(X/R^*,T_{R^*})$. 
	
	Now we show that $\gamma_2$ is proximal. Suppose not, there exists $(y_1,y_2)\in R_{\gamma_2}\setminus \Delta(Y)$ such that 
   $(y_1,y_2)$ is a distal pair. Without loss of generality, we can assume that $(y_1,y_2)$ is a minimal point.
   Since $R_{\pi_3}=R_{\pi_1}\cap R_\psi$ and $(Y,S)=(X/R_{\pi_3},T_{R_{\pi_3}})$, $\gamma_1(y_1)\not= \gamma_1(y_2)$.
   It is clear that $(\theta_1\gamma_1(y_1), \theta_1\gamma_1(y_2))=(\phi\gamma_2(y_1), \phi\gamma_2(y_2))$ and thus $\theta_1\gamma_1(y_1)= \theta_1\gamma_1(y_2)$.
   However  $(\gamma_1(y_1), \gamma_1(y_2))$ is a minimal point which contradict to 
   $\theta_1$ a proximal extension.
   Therefore $\gamma_2$ is proximal.
\end{proof}

We end this section by having the following question:

\begin{ques}
For a minimal system $(X,T)$, how to describe the  $T\times T$-invariant closed equivalence relation on $X$ determined by the maximal pairwise syndetically equicontinuous factor of $(X,T)$?
\end{ques}

\medskip
\noindent\textbf{Acknowledgments.} The authors were supported by NNSF of China (11771264, 11871188) and NSF of Guangdong Province (2018B030306024).
The authors would like to thank the referees for the careful reading and helpful suggestions.

\end{document}